\documentclass[12pt]{article}
\date{September 14, 2026}

\PassOptionsToPackage{dvipsnames,svgnames,x11names}{xcolor}
\usepackage{xcolor}
\definecolor{labelkey}{rgb}{0,0.08,0.45}
\definecolor{refkey}{rgb}{0,0.6,0.0}
\definecolor{Brown}{rgb}{0.45,0.0,0.05}
\definecolor{lime}{rgb}{0.00,0.8,0.0}
\definecolor{lblue}{rgb}{0.5,0.5,0.99}
\definecolor{OliveGreen}{rgb}{0,0.6,0}
\definecolor{tyrianpurple}{rgb}{0.4, 0.01, 0.24}
\definecolor{myseagreen}{HTML}{3FBC9D}
\definecolor{myblue}{rgb}{0.9,0.9,0.98}
\colorlet{hlcyan}{cyan!30}

\usepackage{mathpazo} 
\usepackage{subcaption}
\usepackage[T1]{fontenc}
\usepackage{mathtools}
\usepackage{amssymb}
\usepackage{stmaryrd}
\usepackage{empheq}

\usepackage{amsmath}

\usepackage{hyperref}
\hypersetup{
  colorlinks=true,
  linkcolor=refkey,
  urlcolor=lblue,
  citecolor=magenta
}

\usepackage{amsthm}
\makeatletter
\def\th@plain{%
  \thm@notefont{}%
  \itshape 
}
\def\th@definition{%
  \thm@notefont{}%
  \normalfont 
}

\usepackage[capitalize,nameinlink]{cleveref}

\usepackage{etoolbox}

\newcommand{\currentenv}{}

\AtBeginEnvironment{proposition}{%
  \def\currentenv{Proposition}%
}

\AtBeginEnvironment{example}{%
  \def\currentenv{Example}%
}

\newcommand{\eqoverset}[2]{%
  \mathrel{\mathpalette\eq@overset{{#1}{#2}}}%
}
\newcommand{\eq@overset}[2]{\eq@@overset#1#2}
\newcommand{\eq@@overset}[3]{%
  \vbox{%
    \offinterlineskip
    \ialign{%
      \hfil##\hfil\cr
      \scriptsize $#1#2$\cr
      \noalign{\kern1pt}
      $#1#3$\cr
    }%
  }%
}

\makeatother
\newtheorem{theorem}{Theorem}[section]

\newtheorem{corollary}[theorem]{Corollary}
\newtheorem{proposition}[theorem]{Proposition}

\newtheorem{example}[theorem]{Example}
\newtheorem{fact}[theorem]{Fact}
\newtheorem{remark}[theorem]{Remark}
\crefname{theorem}{Theorem}{Theorems}
\Crefname{theorem}{Theorem}{Theorems}
\crefname{fact}{Fact}{facts}
\Crefname{fact}{Fact}{facts}
\crefname{equation}{}{equations}
\crefname{chapter}{Appendix}{chapters}
\crefname{item}{}{items}
\crefname{enumi}{}{}

\usepackage{nicematrix}
\usepackage{booktabs}
\usepackage{siunitx}
\usepackage{tikz}
\usepackage[color=myseagreen]{todonotes}
\usepackage{soul}
\usepackage{nameref}
\usepackage{graphicx}
\usepackage{float}
\usepackage[shortlabels,inline]{enumitem}
\setlist[enumerate]{nosep}
\usepackage{relsize}
\usepackage{soul}

\usepackage[margin=1in,footskip=0.25in]{geometry}
\makeatletter

\let\orig@label\label

\renewcommand{\label}[1]{%
  \begingroup
  \def\@currentlabelname{}%
  \ifx\current@theorem\relax\else
    \def\@currentlabelname{\current@theorem}%
  \fi
  \ifx\cref@currentlabel\undefined\else
    \let\@currentlabelname\cref@currentlabel
  \fi
  \orig@label{#1}%
  \endgroup
}

\AtBeginEnvironment{theorem}{\def\current@theorem{theorem}}
\AtBeginEnvironment{proposition}{\def\current@theorem{proposition}}
\AtBeginEnvironment{fact}{\def\current@theorem{fact}}
\makeatother

\newcommand{\seppfour}{\setlength{\itemsep}{-4pt}}

\newcommand{\nnn}{\ensuremath{{n\in\mathbb{N}}}}

\newcommand{\menge}[2]{\big\{{#1}~\big|~{#2}\big\}}

\newcommand{\scal}[2]{\left\langle{#1},{#2}\right\rangle}

\newcommand{\RR}{\ensuremath{\mathbb{R}}}

\newcommand{\RP}{\ensuremath{\mathbb{R}_+}}
\newcommand{\RPP}{\ensuremath{\mathbb{R}_{++}}}

\newcommand{\NN}{\ensuremath{\mathbb{N}}}

\newcommand{\dom}{\ensuremath{\operatorname{dom}\,}}

\DeclareMathOperator*{\argmin}{argmin}

\newcommand{\reli}{\ensuremath{\operatorname{ri}}}

\newcommand{\aff}{\ensuremath{\operatorname{aff}}}

\newcommand{\cspan}{\ensuremath{\overline{\operatorname{span}}\,}}

\newcommand{\Id}{\ensuremath{\operatorname{Id}}}

\providecommand{\norm}[1]{\lVert#1\rVert}

\providecommand{\lf}{\left \lfloor}
\providecommand{\rf}{\right \rfloor}

\providecommand{\cone}{\operatorname{cone}}
\providecommand{\ccone}{\overline{\cone}}

\providecommand{\epi}{\operatorname{epi}}

\providecommand{\rec}{\operatorname{rec}}
\providecommand{\spn}{\operatorname{span}}

\providecommand{\amin}{\operatorname{argmin}}

\newcommand{\mybluebox}[1]{\colorbox{myblue}{\hspace{1em}#1\hspace{1em}}}

\author{
  Heinz H.\ Bauschke\thanks{
    Mathematics, University of British Columbia,
    Kelowna, B.C.\ V1V~1V7, Canada. E-mail: \texttt{heinz.bauschke@ubc.ca}.}
  ~~~and~~~
  Tran Thanh Tung\thanks{
    Mathematics, University of British Columbia,
    Kelowna, B.C.\ V1V~1V7, Canada. E-mail: \texttt{tung.tran@ubc.ca}.}
}

\title{\textsf{
The G\"unt\"urk-Thao theorem revisited: \\ 
polyhedral cones and limiting examples
}}

\usepackage[final]{showlabels} 

\begin{document}
\allowdisplaybreaks
\maketitle

\begin{abstract}
    In 2023, G\"unt\"urk and Thao proved that the sequence $(x^{(n)})_\nnn$ generated by random (relaxed) projections drawn from a finite {innately regular collection} of closed subspaces in a real Hilbert space satisfies $\sum_\nnn \|x^{(n)}-x^{(n+1)}\|^\gamma <+\infty$ for all $\gamma>0$.

    We show that the G\"unt\"urk-Thao result is sharp in the sense that it fails for two closed subspaces that is not innately regular.
    We also extend their result to a finite collection of polyhedral cones. Moreover, we construct examples showing the tightness of our extension: indeed, the result fails for a line and a convex set in $\RR^2$, and for a plane and a nonpolyhedral cone in $\RR^3$.
\end{abstract}

{
    \small
    \noindent
    {\bfseries 2020 Mathematics Subject Classification:}
    {Primary 47H09, 65K05, 65F10;
        Secondary 46C05, 90C25.
    }
}

\noindent
{\bfseries Keywords:}
alternating projections,
convex set,
finite-length trajectory,
Hilbert space,
innate regularity,
linear subspace,
polyhedral cone,
projection algorithms,
random products,
reasonable wanderer,
relaxed projections.

\section{Introduction}

Throughout this paper,
\begin{empheq}[box=\mybluebox]{equation}
    \text{$X$ is a real Hilbert space, with inner product $\scal{\cdot}{\cdot}$ and induced norm $\norm{\cdot}$}
\end{empheq}
and
\begin{empheq}[box=\mybluebox]{equation}
    \text{$\mathcal{L}$ is a nonempty finite collection of closed linear subspaces of $X$.}
\end{empheq}

The collection $\mathcal{L}$ is called
\emph{innately regular} if whenever
$L_1,\ldots,L_k$ are drawn from $\mathcal{L}$,
then $\sum_{i=1}^k L_i^\perp$ is closed
(see \cite{HBrandom} and especially \cite[Section~2]{GT} for a nice summary).
Given a nonempty finite collection
$\mathcal{C}$ of closed convex subsets of $X$ and an interval $\Lambda \subseteq~[0,2]$, consider the associated set of relaxed projectors\footnotemark
\begin{empheq}[box=\mybluebox]{equation}
    \mathcal{R}_{\mathcal{C},\Lambda} := \menge{(1-\lambda)\Id+\lambda P_C}{C\in\mathcal{C},\ \lambda\in \Lambda}
\end{empheq}
where $P_C$ is the {metric} projector onto $C$ and $\Id$ is the identity mapping on $X$.
\footnotetext{Given a nonempty closed convex subset $C$ of $X$, we denote by $P_C$ the operator which maps
    $x\in X$ to its unique nearest point in $C$.}

Recently, G\"{u}nt\"{u}rk and Thao proved
the following remarkable result \cite[Theorem~1.1]{GT}:

\begin{fact}[G\"{u}nt\"{u}rk and Thao]
    \label{f:GT}
    Let $\mathcal{L}$ be innately regular, $\lambda\in\left]0,1\right]$, and $x^{(0)}\in X$.
    Generate a sequence $(x^{(n)})_{\nnn}$ in $X$ as follows:
    Given $x^{(n)}$, pick $R_n \in \mathcal{R}_{\mathcal{L},[\lambda,2-\lambda]}$, and update via
    \begin{equation}
        \label{260102a}
        x^{(n+1)} := R_n x^{(n)}.
    \end{equation}
    Then
    \begin{equation}
        \sum_{n=0}^\infty \norm{x^{(n+1)}-x^{(n)}}^\gamma <+\infty\quad \text{for all $\gamma>0$.}
    \end{equation}
\end{fact}

We note that \cite[Theorem~3.3]{HBrandom}
implies that
the strong limit of
$(x^{(n)})_\nnn$ lies in
$\bigcap_{L\in {\mathcal{L}_\infty}} L$, where
$\mathcal{L_\infty} :=
    \menge{L\in\mathcal{L}}{\text{$L$ is used infinitely many times in \cref{260102a}}}$.

\cref{f:GT} is easy to prove if
$\gamma=2$
because by \cite[Lemma 2.4(iv)]{BB1996}, each $R_n$ is $\lambda/(2-\lambda)$-attracting so the result follows from \cite[Example 2.7]{BB1996}. The case $\gamma>2$ then follows directly from the comparison test for infinite series (see \cite[Page 113]{Konrad}). The proof of \cref{f:GT} when $\gamma\in\left]0,2\right[$ is much more involved and based on a clever
induction on the number of subspaces in $\mathcal{L}$.

It is well known that $\sum_{n=0}^\infty \norm{x^{(n+1)}-x^{(n)}}^2<+\infty$; in fact, a sequence
exhibiting this property was called a
\emph{reasonable wanderer}
by Youla and Webb in their influential paper \cite{YoulaWebb}.
If $\sum_{n=0}^\infty \norm{x^{(n+1)}-x^{(n)}}<
    +\infty$, then $(x^{(n)})_\nnn$ has a
\emph{finite-length trajectory}, while if
$\sum_{n=0}^\infty \norm{x^{(n+1)}-x^{(n)}}^0<
    +\infty$, then $(x^{(n)})_\nnn$ converges
in  \emph{finitely many steps}.
Thus, a very nice interpretation of \cref{f:GT} is that
the sequence $(x^{(n)})_\nnn$
\emph{nearly converges in finitely many steps}.

\emph{The goal of this paper is two-fold: We aim to extend G\"unt\"urk and Thao's result from subspaces
    to polyhedral cones, and to provide limiting examples
    showing the tightness of our extension.}

More precisely, we provide the following new result
and limiting examples
complementary to \cref{f:GT}:
\begin{itemize}[itemsep=0pt]
    \item
          We extend \cref{f:GT} from closed linear subspaces to \emph{polyhedral cones}
          in \cref{t:GTpolyconeinfi}, which is our main result.

    \item
          We show that without the innate regularity
          assumption, \cref{f:GT} fails dramatically:
          indeed, we provide two closed linear subspaces $L_1,L_2$
          such that the sequence of \emph{alternating projections} $(x^{(n)})_\nnn$ satisfies
          $\sum_{n=0}^\infty \norm{x^{(n+1)}-x^{(n)}}^\gamma =+\infty$ for all $\gamma\in \left]0,2\right[$
          (see \cref{ex:counterex}).
          This shows that the ``reasonable wanderer'' condition is the best one can hope for in general.
    \item
          We construct in the Euclidean plane a linear subspace
          and a nonempty closed convex set such that the sequence of \emph{alternating projections} $(x^{(n)})_\nnn$ satisfies
          $\sum_{n=0}^\infty \norm{x^{(n+1)}-x^{(n)}}^\gamma =+\infty$ for all $\gamma\in \left]0,2\right[$
          (see \cref{e:notcone}). This highlights the importance of the polyhedral-cone
          assumption in \cref{t:GTpolyconeinfi}.
    \item Building on the work in the Euclidean plane
          (see \cref{e:notcone}), we homogenize that example to
          construct a closed convex cone in $\RR^3$ such that
          the sequence of \emph{alternating projections} $(x^{(n)})_\nnn$, where the other set is a plane, also satisfies
          $\sum_{n=0}^\infty \norm{x^{(n+1)}-x^{(n)}}^\gamma =+\infty$ for all $\gamma\in \left]0,2\right[$
          (see \cref{e:cone}). This highlights the importance of the polyhedrality
          assumption in \cref{t:GTpolyconeinfi}; indeed,
          such an example cannot exist in $\RR$ or $\RR^2$.
\end{itemize}

The rest of the paper is organized as follows:
After discussing faces, we present the extension of \cref{f:GT} to polyhedral cones first in
the finite-dimensional setting (\cref{s:polyhedra}) and
then in the general, possibly infinite-dimensional setting (\cref{s:infdim}). The sharpness of the results is illustrated by the limiting examples provided in \cref{s:counterex}, \cref{s:nonpolyconv}, and \cref{s:last}.

The notation we employ is standard and follows, e.g., \cite{BC2017}
and \cite{Rocky}.

\section{Finite-dimensional extension to polyhedral cones}

\label{s:polyhedra}

In this section, we assume that
\begin{empheq}[box=\mybluebox]{equation}
    \text{
        $X$ is finite-dimensional,
    }
\end{empheq}
i.e., $X$ is a Euclidean space.

\subsection*{Faces and projections of convex sets}

Let $C$ be a nonempty convex subset of $X$.
Recall that a convex subset $F$ of $C$ is a \emph{face} of $C$
if whenever $x,y$ belong to $C$ and
$F \cap \left]x,y\right[\neq \varnothing$,
then $x,y$ both belong to $F$.
Note that $C$ is a face of itself and
denote the collection of all faces of $C$ by $\mathcal{F}(C)$.
We now set
\begin{equation}
    \label{e:F_c}
    (\forall c\in C)\quad
    F_c :=
    \bigcap_{c \in F \in \mathcal{F}(C)} F.
\end{equation}
The following fact is well-known; see, e.g.,
the books by Rockafellar \cite{Rocky} and Webster \cite{Webster}
(which also contain further information on faces):

\begin{fact}
    \label{f:uniqueface}
    Let $C$ be a nonempty convex subset of $X$.
    For each $c\in C$, the set $F_c$ (defined in \cref{e:F_c}) is a face of $C$;
    in fact, $F_c$ is the unique face of $C$ with $c\in\reli F_c$.
    Furthermore,
    \begin{equation}
        C = \bigcup_{F\in \mathcal{F}(C)} \reli F
    \end{equation}
    forms a partition of $C$.
\end{fact}
\begin{proof}
    The fact that $F_c$ is a face follows from
    \cite[Theorem~2.6.5(i)]{Webster} while the
    uniqueness property was stated in \cite[Theorem~2.6.10]{Webster}.
    Because two faces whose relative interiors make a
    nonempty intersection must be equal (see \cite[Corollary~18.1.2]{Rocky}),
    we note that the partition claim follows from
    \cite[Theorem~18.2]{Rocky} or \cite[Theorem~2.6.10]{Webster}.
\end{proof}

Recall that a subset $C$ of $X$ is a cone if $\mathbb{R}_{++}C=C$.
Observe the following:

\begin{fact}
    \label{f:facecone}
    Let $C$ be a nonempty convex cone of $X$, let $F$ be a face of $C$. If $F$ is nonempty, then $F$ is a convex cone.
\end{fact}
\begin{proof}
    Convexity follows directly from the definition of $F$. Since $C$ is a cone, we have
    \begin{equation}
        \left(\forall f\in F\right)(\forall\varepsilon\in\left]0,1\right[)\quad {\varepsilon f}\text{ and }2f \text{ both belong to }C\text{ and }  F\cap\left]{\varepsilon f},2f\right[\ni f.
    \end{equation}
    Hence $\varepsilon f$ and $2f$ belong to $F$, and the result follows.
\end{proof}

Recall that
the \emph{affine hull} of a nonempty subset $S$ of $X$ is
the smallest affine subspace of $X$ containing $S$;
it is denoted by $\aff S$.
We are now in a position to state a recent result
by Fodor and Pintea (see \cite[Theorem~3.1]{Fodor}):

\begin{fact}
    \label{f:Fodor}
    Let $C$ be a nonempty closed convex subset of $X$,
    let $F$ be a face of $C$,
    and let $x\in P_C^{-1}(\reli F)$.
    Then
    \begin{equation}
        P_Cx = P_{\aff F}x.
    \end{equation}
\end{fact}

\begin{corollary}
    \label{c:projandface}
    Let $C$ be a nonempty closed convex subset of $X$,
    and let $x\in X$.
    Then there exists a face $F$ of $C$ such that
    \begin{equation}
        P_Cx = P_{\aff F}x.
    \end{equation}
\end{corollary}
\begin{proof}
    Set $c := P_Cx$.
    Then $c\in C$ and obviously $x\in P_C^{-1}(c)$.
    In view of \cref{f:uniqueface},
    the face $F_c$ defined in \cref{e:F_c}
    satisfies $c\in\reli F_c$.
    Set $F := F_c$.
    Altogether, we deduce that $x\in P_C^{-1}(\reli F)$,
    and \cref{f:Fodor} implies $P_Cx = P_{\aff F}x$.
\end{proof}

\begin{corollary}
    \label{c:projandfacecone}
    Let $C$ be a nonempty closed convex cone of $X$,
    and let $x\in X$.
    Then there exists a face $F$ of $C$ such that
    \begin{equation}
        P_Cx = P_{\spn(F)}x.
    \end{equation}
\end{corollary}
\begin{proof}
    Combine \cref{f:facecone} and \cref{c:projandface}.
\end{proof}

\subsection*{Polyhedral cones and the extension}

Recall that a subset $C$ of $X$ is a \emph{polyhedral cone}
if it is both a polyhedral set\footnotemark and a cone.
\footnotetext{A subset $P$ of $X$ is \emph{polyhedral}
    if it is the intersection of finitely
    many closed halfspaces.}
Note that a linear subspace of $X$ is a polyhedral cone;
in the general infinite-dimensional setting
a closed linear subspace is a polyhedral cone
if and only if it is finite-codimensional.

Polyhedral sets are precisely those closed convex sets  with finitely many faces (see \cite[Theorem~19.1]{Rocky} or \cite[Theorems~3.2.2 and 3.2.3]{Webster}):

\begin{fact}
    \label{f:polychar}
    Let $C$ be a nonempty closed convex subset of $X$.
    Then $C$ is a polyhedron if and only if it has finitely many faces.
\end{fact}

We are now ready to extend \cref{f:GT} from
linear subspaces to polyhedral cones:

\begin{theorem}[polyhedral cones in Euclidean space]
    \label{t:GTpolycone}
    Let $\mathcal{C}$ be a nonempty finite collection of nonempty polyhedral cones in $X$.
    Let $\lambda\in\left]0,1\right]$ and $x^{(0)}\in X$.
    Generate a sequence $(x^{(n)})_{\nnn}$ in $X$ as follows:
    Given the current term $x^{(n)}$, pick $R_n \in \mathcal{R}_{\mathcal{C},[\lambda,2-\lambda]}$, and update via
    \begin{equation}
        x^{(n+1)} := R_n x^{(n)}.
    \end{equation}
    Then
    \begin{equation}
        \sum_{n=0}^\infty \norm{x^{(n+1)}-x^{(n)}}^\gamma <+\infty\quad \text{for all $\gamma>0$.}
    \end{equation}
\end{theorem}
\begin{proof}
    By \cref{f:polychar},
    each polyhedral cone $C\in\mathcal{C}$ has finitely many faces, i.e.,
    $\mathcal{F}(C)$ is finite.
    Because $\mathcal{C}$ is finite, so is
    $\mathcal{F} := \bigcup_{C\in\mathcal{C}} \mathcal{F}(C)$
    and also
    $\{\spn(F)\}_{F\in\mathcal{F}}$.
    Now suppose that $\mathcal{L} = \{\spn(F)\}_{F\in\mathcal{F}}$.
    If $x\in X$ and $C\in\mathcal{C}$, then
    \cref{c:projandfacecone}
    yields $P_Cx = P_Lx$, for some $L\in\mathcal{L}$ depending
    on $C$ and $x$.
    It follows that the sequence $(x^{(n)})_\nnn$ can be viewed
    as being generated by $\mathcal{L}$. Since $X$ is finite-dimensional, we obtain that $\mathcal{L}$ is obviously innately regular (because all linear subspaces
    are automatically closed).
    Therefore, the result follows from \cref{f:GT}.
\end{proof}

\section{General extension to polyhedral cones}

\label{s:infdim}

In this section, we return to our general assumption on $X$, i.e.,
\begin{empheq}[box=\mybluebox]{equation}
    \text{
        $X$ is finite or infinite-dimensional.
    }
\end{empheq}

The following representation result for a convex polyhedral set and
its projection is useful (see \cite[Theorem~3.3.16]{HBthesis}):

\begin{fact}
    \label{f:thesis}
    Let $C$ be a nonempty polyhedral subset of $X$, represented
    as
    \begin{equation}
        C = \menge{x\in X}{(\forall i\in I)\;\;\scal{a_i}{x}\leq \beta_i},
    \end{equation}
    where $I$ is a finite index set, each $a_i$ belongs to
    $X\smallsetminus\{0\}$ and each $\beta_i$ is a real number.
    Let $K$ be a closed linear subspace of $\bigcap_{i\in I}\ker a_i$,
    and set $D := C\cap K^\perp$.
    Then $D$ is a nonempty polyhedral subset of $K^\perp$ and
    \begin{equation}
        \label{e:thesis}
        P_C = P_K + P_DP_{K^\perp}.
    \end{equation}
    Note that if $K = \bigcap_{i\in I}\ker a_i$,
    then $K^\perp$ is finite-dimensional.
\end{fact}

We are now ready to extend \cref{t:GTpolycone} to
general Hilbert spaces:

\begin{theorem}[main result]
    \label{t:GTpolyconeinfi}
    Let $\mathcal{C}$ be a nonempty finite collection of
    nonempty polyhedral cones in $X$.
    Let $\lambda\in]0,1]$ and $x^{(0)}\in X$.
    Generate a sequence $(x^{(n)})_{\nnn}$ in $X$ as follows:
    Given a current term $x^{(n)}$, pick $R_n \in \mathcal{R}_{\mathcal{C},[\lambda,2-\lambda]}$, and update via
    \begin{equation}
        x^{(n+1)} := R_n x^{(n)}.
    \end{equation}
    Then
    \begin{equation}
        \sum_{n=0}^\infty \norm{x^{(n+1)}-x^{(n)}}^\gamma <+\infty\quad \text{for all $\gamma>0$.}
    \end{equation}
\end{theorem}
\begin{proof}
    For each $C\in\mathcal{C}$, we pick
    a finite-codimensional closed
    linear subspace $K_C$ of $X$ as in
    \cref{f:thesis}.
    Next, we set $K := \bigcap_{C\in\mathcal{C}} K_C$.
    Then $K$ is a closed linear subspace of $X$, still
    finite-codimensional.
    By \cref{f:thesis} again, we have
    \begin{equation}
        \label{e:boom}
        (\forall C\in\mathcal{C})\quad
        P_C = P_K + P_{D_C}P_{K^\perp},
    \end{equation}
    where $D_C = C \cap K^\perp$ is a nonempty polyhedral subset of
    the finite-dimensional space $K^\perp$. Note that $D_C$ is also a polyhedral cone because the intersection of a polyhedral cone and a finite-dimensional linear subspace is still polyhedral.

    The starting point $x^{(0)}$ is decomposed as
    $x^{(0)} = k^{(0)} + y^{(0)}$, where
    $k^{(0)} := P_Kx^{(0)} \in K$ and $y^{(0)} := P_{K^\perp}x^{(0)} \in K^\perp$.
    Observe that \cref{e:boom} implies that
    $(x^{(n)})_\nnn = (k^{(0)})_\nnn +(y^{(n)})_\nnn$,
    where $(y^{(n)})_\nnn$
    is generated using the finite collection of polyhedral cones
    $\{D_C\}_{C\in\mathcal{C}}$ in the finite-dimensional space $K^\perp$. Observe that
    \begin{equation}
        \left(\forall n\in\mathbb{N}\right)\quad \norm{x^{(n+1)}-x^{(n)}}=\norm{y^{(n+1)}-y^{(n)}}.
    \end{equation}
    The result then follows from \cref{t:GTpolycone}.
\end{proof}

\section{Impossibility to drop the innate regularity assumption}\label{s:counterex}

The following example shows that we cannot drop the innate regularity assumption in \cref{f:GT}:

\begin{example}[without innate regularity, G\"unt\"urk-Thao fails]
    \label{ex:counterex}
    Suppose that $X=\ell_2(\mathbb{N})$ with orthonormal basis $(e_n)_{n\in\mathbb{N}}$. Suppose that $(\theta_n)_{n\in\mathbb{N}}$ is a decreasing sequence of angles in $]0,\pi/2[$ with $\theta_n\to0$. Consider the collection $\mathcal{L} = \{L_1,L_2\}$
    of two closed linear subspaces where $L_1:=\cspan\{\cos(\theta_k)e_{2k-1}+\sin(\theta_k)e_{2k}\mid k\in\mathbb{N}\smallsetminus\{0\}\}$ and $L_2:=\cspan\{e_{2k+1}\mid k\in\mathbb{N}\}$. We have
    \begin{enumerate}[ref=\currentenv~\theexample(\roman*)]
        \item\label{ex:counterexi} $P_{L_1}x=\sum_{k\in\mathbb{N}\smallsetminus\{0\}}\left(x_{2k-1}\cos(\theta_k)+x_{2k}\sin(\theta_k)\right)\left(e_{2k-1}\cos(\theta_k)+e_{2k}\sin(\theta_k)\right)$.
        \item\label{ex:counterexii} $P_{L_2}x=(0,x_1,0,x_3,0,x_5,\dots)$.
        \item\label{ex:counterexiii} $\mathcal{L}$ is not innately regular.
    \end{enumerate}
    Set  $x^{(0)}:=\left(0,x_1,0,x_3,0,x_5,\dots\right)\in\ell_2(\mathbb{N})$ and generate the sequence of alternating projections via
    \begin{equation}
        x^{(2n+1)} := P_{L_1}x^{(2n)}\text{ and }x^{(2n+2)} := P_{L_2}x^{(2n+1)}.\label{20252610a}
    \end{equation}
    Then:
    \begin{enumerate}[ref=\currentenv~\theexample(\roman*)]
        \setcounter{enumi}{3}
        \item\label{ex:counterexiv} For all $n\in\mathbb{N}\smallsetminus\{0\}$, the iterates have the following form
              \begin{equation}
                  \left(\forall k\in\mathbb{N}\smallsetminus\{0\}\right)\quad\begin{cases}
                      x^{(2n-1)}_0=0,                                  \\
                      x^{(2n-1)}_{2k-1}=x_{2k-1}\cos^{2n}(\theta_{k}), \\
                      x^{(2n-1)}_{2k}=x_{2k-1}\cos^{2n-1}(\theta_{k})\sin(\theta_{k}),\label{20252610b}
                  \end{cases}
              \end{equation}
              and
              \begin{equation}
                  \left(\forall k\in\mathbb{N}\smallsetminus\{0\}\right)\quad\begin{cases}
                      x^{(2n)}_{2k-2}=0, \\
                      x^{(2n)}_{2k-1}=x_{2k-1}\cos^{2n}(\theta_{k}).\label{20252610c}
                  \end{cases}
              \end{equation}
        \item\label{ex:counterexv} For all $n\in\mathbb{N}\smallsetminus\{0\}$: $\norm{x^{(2n)}-x^{(2n-1)}}^2=\sum_{k=1}^\infty x^2_{2k-1}\cos^{4n}(\theta_{k})\left(\cos^{-2}(\theta_{k})-1\right)$.
    \end{enumerate}
    Suppose $\theta_0\in\left]0,\pi/2\right[$, $\theta_k:=\arccos\Big(\big(\frac{k}{k+1}\big)^{\frac{1}{4e^k}}\Big)$, and $x_{2k-1}:=\frac{1}{k^2}$ for all $k\in\mathbb{N}\smallsetminus\{0\}$. Then:
    \begin{enumerate}[ref=\currentenv~\theexample(\roman*)]
        \setcounter{enumi}{5}
        \item\label{ex:counterexvi} For all $\gamma\in\left]0,2\right[$: $\displaystyle \sum_{n=0}^\infty\norm{x^{(n+1)}-x^{(n)}}^\gamma=+\infty$.
    \end{enumerate}
\end{example}
\begin{proof}
    {\it(i)-(iii)}: See \cite[Example~4.1]{BE2005}.

        {\it(iv)}: We will prove this by induction. For the base case $n=1$, by using \cref{ex:counterexi}, \cref{ex:counterexii}, and \cref{20252610a}, we obtain
    \begin{subequations}
        \begin{align}
                            & x^{(1)}=P_{L_1}x^{(0)}=\sum_{k\in\mathbb{N}\smallsetminus\{0\}}x_{2k-1}\cos(\theta_k)\left(e_{2k-1}\cos(\theta_k)+e_{2k}\sin(\theta_k)\right) \\
            \Leftrightarrow & \left(\forall k\in\mathbb{N}\smallsetminus\{0\}\right)\quad\begin{cases}
                                                                                             x^{(1)}_0=0,                                 \\
                                                                                             x^{(1)}_{2k-1}=x_{2k-1}\cos^{2}(\theta_{k}), \\
                                                                                             x^{(1)}_{2k}=x_{2k-1}\cos(\theta_{k})\sin(\theta_{k}),
                                                                                         \end{cases}
        \end{align}
    \end{subequations}
    and
    \begin{equation}
        \left(\forall k\in\mathbb{N}\smallsetminus\{0\}\right)\quad\begin{cases}
            x^{(2)}_{2k-2}=0, \\
            x^{(2)}_{2k-1}=x_{2k-1}\cos^{2}(\theta_{k}).
        \end{cases}
    \end{equation}
    Assume that \cref{20252610b} and \cref{20252610c} are true for $n=m\in\mathbb{N}\smallsetminus\{0\}$. In particular, for all $k\in\mathbb{N}\smallsetminus\{0\}$, we have
    \begin{equation}
        \begin{cases}
            x^{(2m)}_{2k-2}=0, \\
            x^{(2m)}_{2k-1}=x_{2k-1}\cos^{2m}(\theta_{k}),
        \end{cases}
    \end{equation}
    which implies
    \begin{subequations}
        \begin{align}
                            & x^{(2m+1)}=P_{L_1}x^{(2m)}=\sum_{k\in\mathbb{N}\smallsetminus\{0\}}x_{2k-1}\cos^{2m+1}(\theta_k)\left(e_{2k-1}\cos(\theta_k)+e_{2k}\sin(\theta_k)\right) \\
            \Leftrightarrow & \left(\forall k\in\mathbb{N}\smallsetminus\{0\}\right)\quad\begin{cases}
                                                                                             x^{(2m+1)}_0=0,                                    \\
                                                                                             x^{(2m+1)}_{2k-1}=x_{2k-1}\cos^{2m+2}(\theta_{k}), \\
                                                                                             x^{(2m+1)}_{2k}=x_{2k-1}\cos^{2m+1}(\theta_{k})\sin(\theta_{k}).
                                                                                         \end{cases}
        \end{align}
    \end{subequations}
    By \cref{ex:counterexii}, we are done.

        {\it(v)}: From \cref{20252610b} and \cref{20252610c}, we get that
    \begin{subequations}
        \begin{align}
            (\forall n\in\mathbb{N}\smallsetminus\{0\})\quad\norm{x^{(2n)}-x^{(2n-1)}}^2 & =\sum_{k=1}^\infty x^2_{2k-1}\cos^{4n-2}(\theta_k)\sin^2(\theta_k)                   \\
                                                                                         & =\sum_{k=1}^\infty x^2_{2k-1}\cos^{4n}(\theta_{k})\big(\cos^{-2}(\theta_{k})-1\big).
        \end{align}
    \end{subequations}

    {\it(vi)}: From the choice of $x^{(0)}$, $(\theta_k)_{k\in\mathbb{N}}$, and \cref{ex:counterexv}, we obtain
    \begin{equation}
        (\forall n\in\mathbb{N}\smallsetminus\{0\})\quad\norm{x^{(2n)}-x^{(2n-1)}}^2=\sum_{k=1}^\infty\frac{1}{k^4}\Big(\frac{k}{k+1}\Big)^{\frac{n}{e^k}}\bigg(\Big(1+\frac{1}{k}\Big)^{\frac{1}{2e^k}}-1\bigg).\label{20252610d}
    \end{equation}
    Define $\delta_n:=\ln(n)-\lf\ln(n)\rf$ for all $n\geq3$. Observe that the sequence $(\delta_n)_{n\geq3}$ lies in $[0,1[$. Since in \cref{20252610d} all components of the infinite series are positive, for all \(n \geq3\) we obtain
    \begin{subequations}
        \begin{align}
            \norm{x^{(2n)}-x^{(2n-1)}}^2 & \geq\frac{1}{(\ln(n)-\delta_n)^4}\left(\frac{\ln(n)-\delta_n}{\ln(n)-\delta_n+1}\right)^{e^{\delta_n}}\bigg(
            \Big(1+\frac{1}{\ln(n)-\delta_n}\Big)^{\frac{e^{\delta_n}}{2n}}-1\bigg)                                                                                                                                                  \\
                                         & =\frac{1}{(\ln(n)-\delta_n)^4}\left(\frac{\ln(n)-\delta_n}{\ln(n)-\delta_n+1}\right)^{e^{\delta_n}}\left(e^{\frac{e^{\delta_n}}{2n}\ln\left(1+\frac{1}{\ln(n)-\delta_n}\right)}-1\right).
        \end{align}
    \end{subequations}
    Using the inequalities
    \begin{equation}
        (\forall x>0)\quad e^x-1\geq x\quad\text{and}\quad\ln\Big(1+\frac{1}{x}\Big)\geq\frac{1}{x+1},
    \end{equation}
    we get that
    \begin{subequations}
        \begin{align}
            (\forall n\geq 3)\quad
               & \norm{x^{(2n)}-x^{(2n-1)}}^2                                                                   \\            &\geq\frac{1}{(\ln(n)-\delta_n)^4}\left(\frac{\ln(n)-\delta_n}{\ln(n)-\delta_n+1}\right)^{e^{\delta_n}}\frac{e^{\delta_n}}{2n}\frac{1}{\ln(n)-\delta_n+1}\\
               & =\frac{e^{\delta_n}}{2n(\ln(n)-\delta_n)^{4-e^{\delta_n}}(\ln(n)-\delta_n+1)^{1+e^{\delta_n}}} \\
            \  & \geq\frac{1}{2n\left(\ln(n)+1\right)^5},
        \end{align}
    \end{subequations}
    where the last inequality follows from
    $0<\ln(n)-\delta_n<\ln(n)+1$.
    This implies
    \begin{equation}
        \sum_{n=0}^\infty\norm{x^{(n+1)}-x^{(n)}}^\gamma\geq\sum_{n=3}^\infty\norm{x^{(2n)}-x^{(2n-1)}}^\gamma\geq\sum_{n=3}^\infty\frac{1}{(2n)^{\frac{\gamma}{2}}(\ln(n)+1)^{\frac{5\gamma}{2}}}.
    \end{equation}
    Since $\tfrac{(2n)^{\gamma/2}(\ln(n)+1)^{5\gamma/2}}{n^{\gamma/2}\ln^{5\gamma/2}(n)}\to2^{\gamma/2}$, $\gamma\in]0,2[$, and $\sum_{n\geq3}\tfrac{1}{n^p\ln^q(n)}=+\infty$ for all $p\in\left]0,1\right[$ and $q>0$ (see \cite[80--Theorem, Page~124]{Konrad}), we are done.
\end{proof}

\section{Impossibility to extend from polyhedral cones
  to closed convex sets}\label{s:nonpolyconv}

In this section, we will construct a family of collections, each containing one polyhedral cone and one set that is neither a cone nor a polyhedron, such that \cref{t:GTpolyconeinfi} fails.

    {\begin{remark}
            A simple failure occurs already for two disjoint affine subspaces. Indeed, let
            \begin{equation}
                A:=\RR\times\{0\}
                \quad\text{and}\quad
                B:=\{(0,1)\}.
            \end{equation}
            Starting from $x^{(0)}=(0,0)$, alternating projections give
            $x^{(2n)}=(0,0)$ and $x^{(2n+1)}=(0,1)$; hence
            \begin{equation}
                \sum_{n=0}^{\infty}
                \norm{x^{(n+1)}-x^{(n)}}^\gamma=+\infty
                \quad\text{for every $\gamma>0$.}
            \end{equation}
            More generally, \cite[Theorem~4.8]{BB1994} shows that this happens for any two nonempty closed convex sets at positive distance. In contrast, the two sets constructed below intersect, and the corresponding alternating-projection sequence converges to their unique common point, yet the G\"unt\"urk-Thao conclusion still fails.
        \end{remark}}

We consider the case where our collection consists of
\begin{equation}
    L:=\mathbb{R}\times\{0\}\quad\text{and}\quad C:=\epi f,
\end{equation}
where $f:\mathbb{R}\to\left]-\infty,+\infty\right]$ is
\begin{equation}
    \text{convex, differentiable on $[0,\delta]$, $\delta\in\mathbb{R}_{++}$, lower semicontinuous, and proper,}
\end{equation}
with
\begin{equation}
    f\text{ even, }f(0)=0,f>0\text{ otherwise, and }f'(0)=0.
\end{equation}

The projection onto $L$ is simple:
\begin{equation}
    P_L\colon \RR^2\to\RR^2 \colon (x,r)\mapsto(x,0).\label{251213a}
\end{equation}

The following fact gives us the projection onto $C$ (see \cite[Fact~4.1 and Corollary~4.2]{HM2016}):
\begin{fact}\label{f:epi}
    Let $(x,r)\in(\dom f\times\RR)\smallsetminus\epi f$. Then $P_{\epi f}(x,r)=(y,f(y))$, where $y$ satisfies $x\in y+(f(y)-r)\partial f(y)$. Moreover, $y=0$ if $x=0$, and $y$ lies strictly between $x$ and 0 otherwise.
\end{fact}

Let $x^{(0)}=(u_0,0)\in\left]0,\delta\right[\times\{0\}\subseteq L$. Generate the sequence of alternating projections via
\begin{equation}
    x^{(2n+1)} := P_{C}x^{(2n)}\text{ and }
    x^{(2n+2)} := P_{L}x^{(2n+1)}.\label{251213b}
\end{equation}
An instance of the trajectory is illustrated in \cref{fig:placeholder}.

\begin{figure}[H]
    \centering
    \includegraphics[width=\linewidth]{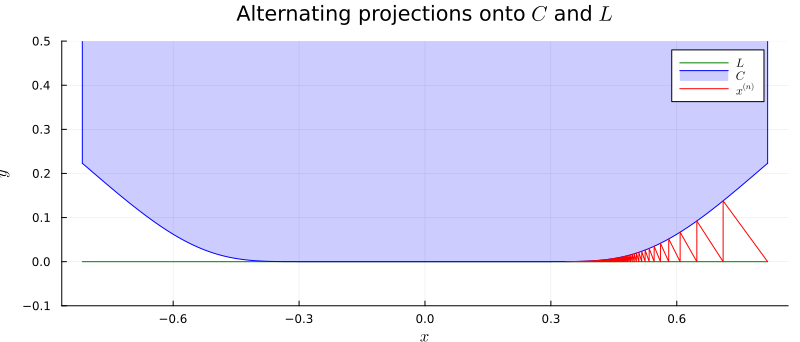}
    \caption{The trajectory of \cref{251213b} for $f(x)=\exp(- x^{-2})$ with domain $\bigl[-\sqrt{2/3},\sqrt{2/3}\bigr]$.}
    \label{fig:placeholder}
\end{figure}

\begin{proposition}\label{p:R2sequence1}
    The iterates have the following form
    \begin{equation}
        (\forall n\in\mathbb{N}\smallsetminus\{0\})\quad x^{(2n-1)}=(u_n,f(u_n))\quad\text{and}\quad x^{(2n)}=(u_n,0),
    \end{equation}
    where
    \begin{equation}
        (\forall n\in\mathbb{N}\smallsetminus\{0\})\quad u_{n-1}=u_n+f(u_n)f'(u_n),
    \end{equation}
    and $(u_n)_{n\in\mathbb{N}}$ is a positive sequence strictly decreasing to $0$.
\end{proposition}
\begin{proof}
    Combine \cref{251213a}, \cref{f:epi}, \cref{251213b}, and \cite[Theorem 1]{Bregman1965}.
\end{proof}

We first analyze the asymptotic behavior of $(u_n)_{n\in\mathbb{N}}$.

\begin{proposition}\label{p:R2sequence2}
    Suppose there exist $q\in\mathbb{R}$ and $p,\alpha\in\mathbb{R}_{++}$ such that
    \begin{equation}
        \lim_{x\searrow0}\frac{f(x)f'(x)}{x^q\exp(-\alpha x^{-p})}=c_{q,\alpha,p}>0.\label{251213c}
    \end{equation}
    Then $\displaystyle \frac{u_n}{\big(\frac{\alpha}{\ln(n)}\big)^{1/p}}\to1$ as $n\to\infty$.
\end{proposition}

\begin{proof}
    By \cref{p:R2sequence1}, we obtain $u_n\searrow0^+$ and
    \begin{equation}
        (\forall n\in\mathbb{N}\smallsetminus\{0\})\quad u_{n-1}=u_n+f(u_n)f'(u_n).\label{251213e}
    \end{equation}
    This together with \cref{251213c} yields
    \begin{equation}
        (\forall n\in\mathbb{N}\smallsetminus\{0\})\quad \frac{f(u_n)f'(u_n)}{u_n^q\exp\big(-\alpha u_n^{-p}\big)}=\frac{u_{n-1}-u_n}{u_n^q\exp\big(-\alpha u_n^{-p}\big)}.\label{251213f}
    \end{equation}
    By checking the derivative,
    $\displaystyle \frac{1}{x^q\exp(-\alpha x^{-p})}$ is decreasing on $\mathbb{R}_{++}$ if $q\geq0$ and
    on $\Bigl]0,\left(\frac{-q}{\alpha p}\right)^{-1/p}\Bigr[$ if $q<0$. WLOG, assume that $u_0$ lies inside the decreasing interval. Since $u_{n-1}>u_n$, we get
    \begin{equation}
        (\forall n\in\mathbb{N}\smallsetminus\{0\})\;\frac{u_{n-1}-u_n}{u_n^q\exp\big(-\alpha u_n^{-p}\big)}\geq\int_{u_n}^{u_{n-1}}\frac{dx}{x^q\exp(-\alpha x^{-p})}\geq\frac{u_{n-1}-u_n}{u_{n-1}^q\exp\big(-\alpha u_{n-1}^{-p}\big)}.\label{251213i}
    \end{equation}
    It follows from $u_n\searrow0^+$, \cref{251213c}, and \cref{251213f} that
    \begin{equation}
        \lim_{n\to\infty}\frac{u_{n-1}-u_n}{u_n^q\exp\big(-\alpha u_n^{-p}\big)}=c_{q,\alpha,p}>0.\label{251213g}
    \end{equation}
    Observe that \cref{251213e}
    and \cref{251213c} imply
    \begin{subequations}
        \label{251213h}
        \begin{align}
            \lim_{n\to\infty}\frac{u_{n-1}}{u_n} & =1+\lim_{n\to\infty}\frac{f(u_n)f'(u_n)}{u_n}                                                                                \\
                                                 & =1+\lim_{n\to\infty}\frac{f(u_n)f'(u_n)}{u_n^q\exp\big(-\alpha u_n^{-p}\big)}\frac{u_n^q\exp\big(-\alpha u_n^{-p}\big)}{u_n} \\ &=1+c_{q,\alpha,p}\lim_{n\to\infty}\frac{u_n^{q-1}}{\exp\big(\alpha u_n^{-p}\big)}\\
                                                 & =1.
        \end{align}
    \end{subequations}
    By combining \cref{251213g} and \cref{251213h}, we obtain
    \begin{subequations}
        \begin{align}
            \lim_{n\to\infty}\frac{u_{n-1}-u_n}{u_{n-1}^q\exp\big(-\alpha u_{n-1}^{-p}\big)} & =\lim_{n\to\infty}\frac{u_{n-1}-u_n}{u_{n}^q\exp\big(-\alpha u_{n}^{-p}\big)}\frac{u_{n}^q\exp\big(-\alpha u_{n}^{-p}\big)}{u_{n-1}^q\exp\big(-\alpha u_{n-1}^{-p}\big)} \\
                                                                                             & =c_{q,\alpha,p}\lim_{n\to\infty}\frac{u_n^q}{u_{n-1}^q}\exp\bigg(-\alpha u_n^{-p}\Big(1-\frac{u_{n-1}^{-p}}{u_n^{-p}}\Big)\bigg)                                         \\
                                                                                             & =c_{q,\alpha,p}\exp\bigg(-\alpha\lim_{n\to\infty} u_n^{-p}\Big(1-\frac{u_{n}^{p}}{u_{n-1}^{p}}\Big)\bigg).
        \end{align}
    \end{subequations}
    By the Taylor expansion of $1-x^p$ at $x=1$, we get that $\frac{1-x^p}{p(1-x)}\to 1$ as $x\to1$. This together with \cref{251213h} yields $\frac{1-\frac{u_{n}^{p}}{u_{n-1}^{p}}}{p\left(1-\frac{u_n}{u_{n-1}}\right)}\to1$. Hence,
    \begin{subequations}
        \label{251312j}
        \begin{align}
            \lim_{n\to\infty}\frac{u_{n-1}-u_n}{u_{n-1}^q\exp\big(-\alpha u_{n-1}^{-p}\big)}=\:\, & c_{q,\alpha,p}\exp\bigg(-\alpha\lim_{n\to\infty} u_n^{-p}\Big(1-\frac{u_{n}^{p}}{u_{n-1}^{p}}\Big)\bigg)                       \\
            =\:\,                                                                                 & c_{q,\alpha,p}\exp\left(-\alpha p\lim_{n\to\infty} u_n^{-p}\Big(1-\frac{u_{n}}{u_{n-1}}\Big)\right)                            \\
            \eqoverset{\text{\cref{251213h}}}{=}
            \                                                                                     & c_{q,\alpha,p}\exp\bigg(-\alpha p\lim_{n\to\infty} \frac{u_{n-1}-u_{n}}{u_{n}^{p+1}}\bigg)                                     \\
            \eqoverset{\text{\cref{251213g}}}{=}
            \                                                                                     & c_{q,\alpha,p}\exp\bigg(-\alpha pc_{q,\alpha,p}\lim_{n\to\infty} \frac{u_n^q\exp\big(-\alpha u_n^{-p}\big)}{u_{n}^{p+1}}\bigg) \\
            =\:\,                                                                                 & c_{q,\alpha,p}\exp\left(0\right)                                                                                               \\
            =\:\,                                                                                 & c_{q,\alpha,p}.
        \end{align}
    \end{subequations}
    From \cref{251213i}, \cref{251213g}, and \cref{251312j}, we obtain
    \begin{equation}
        \label{260101a}
        \lim_{n\to\infty}\int_{u_n}^{u_{n-1}}\frac{dx}{x^q\exp(-\alpha x^{-p})}=c_{q,\alpha,p}.
    \end{equation}
    This implies
    \begin{equation}
        c_{q,\alpha,p}=\lim_{n\to\infty}\frac{1}{n}\sum_{k=1}^n\int_{u_k}^{u_{k-1}}\frac{dx}{x^q\exp(-\alpha x^{-p})}=\lim_{n\to\infty}\frac{1}{n}\int_{u_n}^{u_{0}}\frac{dx}{x^q\exp(-\alpha x^{-p})}.\label{251214a}
    \end{equation}
    By a change of variable $t:=x^{-p}$, we obtain
    \begin{subequations}
        \label{251214b}
        \begin{align}
            \int_{u_n}^{u_{0}}\frac{dx}{x^q\exp(-\alpha x^{-p})} & =\int_{u_n}^{u_{0}}x^{-q}\exp(\alpha x^{-p})dx                                 \\
                                                                 & =\frac{1}{p}\int_{u_0^{-p}}^{u_{n}^{-p}}t^{\frac{-q+p+1}{-p}}\exp(\alpha t)dt.
        \end{align}
    \end{subequations}
    Set $F(x):=\int_{u_0^{-p}}^{x}t^{\frac{-q+p+1}{-p}}\exp(\alpha t)dt$. By the Fundamental Theorem of Calculus, we get $F'(x)=x^{\frac{-q+p+1}{-p}}\exp(\alpha x)$. Consequently, $F''(x)=\left(\frac{-q+p+1}{-p}x^{\frac{-q+2p+1}{-p}}+\alpha x^{\frac{-q+p+1}{-p}}\right)\exp(\alpha x)$.
    Since $t^{\frac{-q+p+1}{-p}}\geq\exp(-\alpha t/2)$ for $t$ large enough, we have $F(x)\to+\infty$ and $F'(x)\to+\infty$ as $x\to+\infty$. Hence, by L'Hôpital's rule, we obtain
    \begin{equation}
        \lim_{x\to+\infty}\frac{F(x)}{F'(x)}
        =\lim_{x\to+\infty}\frac{F'(x)}{F''(x)}
        = \lim_{x\to+\infty}
        \frac{1}{\frac{-q+p+1}{-p}x^{-1}+\alpha}=
        \frac{1}{\alpha}>0.\label{251214c}
    \end{equation}
    Consequently, we get that
    \begin{subequations}
        \begin{align}
            c_{q,\alpha,p} & =\lim_{n\to\infty}\frac{1}{n}\int_{u_n}^{u_{0}}\frac{dx}{x^q\exp(-\alpha x^{-p})}                         \\
                           & =\lim_{n\to\infty}\frac{1}{n}\frac{1}{p}\int_{u_0^{-p}}^{u_{n}^{-p}}t^{\frac{-q+p+1}{-p}}\exp(\alpha t)dt \\
                           & =\lim_{n\to\infty}\frac{1}{n}\frac{1}{p}\frac{1}{\alpha}F'(u_n^{-p})                                      \\
                           & =\frac{1}{p\alpha}\lim_{n\to\infty}\frac{1}{n}u_n^{-q+p+1}\exp\big(\alpha u_n^{-p}\big),
        \end{align}
    \end{subequations}
    This implies
    \begin{subequations}
        \begin{align}
            \lim_{n\to\infty}\frac{\ln\big(u_n^{-q+p+1}\exp(\alpha u_n^{-p})\big)}{\ln(n)} & =\lim_{n\to\infty}\frac{\ln\Big(\frac{u_n^{-q+p+1}\exp(\alpha u_n^{-p})}{n}\Big)+\ln(n)}{\ln(n)} \\
                                                                                           & =0+\lim_{n\to\infty}\frac{\ln(n)}{\ln(n)}                                                        \\
                                                                                           & =1.
        \end{align}
    \end{subequations}
    Since
    \begin{equation}
        \lim_{n\to\infty}\frac{\ln\big(u_n^{-q+p+1}
            \exp(\alpha u_n^{-p})\big)}{\alpha u_n^{-p}}=\lim_{n\to\infty}\frac{(-q+p+1)\ln\left(u_n\right)+\alpha u_n^{-p}}{\alpha u_n^{-p}}=1,
    \end{equation}
    we obtain $\frac{\alpha u_n^{-p}}{\ln(n)}\to 1$, i.e., $\frac{u_n}{\big(\frac{\alpha}{\ln(n)}\big)^{1/p}}\to1$ as $n\to\infty$.
\end{proof}

We are now ready to announce our counterexample:
\begin{example}[G\"unt\"urk-Thao fails for a line and a
        convex set in $\RR^2$]
    \label{e:notcone}
    Let $\beta\in \mathbb{R}_{++}$ and
    $r\in \{2,4,6,8,\ldots\}$.
    Consider the function $f(x)=\exp(-\beta x^{-r})$ with domain $\bigl[-\left(\beta r/(r+1)\right)^{1/r},\left(\beta r/(r+1)\right)^{1/r}\bigr]$.
    The derivative of $f$ is
    \begin{equation}
        f'(x)=\beta rx^{-r-1}\exp(-\beta x^{-r}).
    \end{equation}
    Observe that the assumption of \cref{p:R2sequence2} is satisfied with $q=-r-1$, $\alpha=2\beta$, $p=r$, and $c_{q,\alpha,p}=r\beta$. Hence, we obtain
    \begin{equation}
        1=\lim_{n\to\infty}\frac{u_n}{\left(\frac{\alpha}{\ln(n)}\right)^{1/p}}=\lim_{n\to\infty}\frac{u_n}{\left(\frac{2\beta}{\ln(n)}\right)^{1/r}},
    \end{equation}
    which implies
    \begin{equation}
        (\forall\varepsilon\in\left]0,2\right[)(\exists m_\varepsilon\in\mathbb{N})(\forall n\geq m_\varepsilon)\quad u_n\geq\left(\left(1-\frac{\varepsilon}{2}\right)\left(\frac{2\beta}{\ln(n)}\right)\right)^{1/r}.\label{260106a}
    \end{equation}
    This combined with \cref{p:R2sequence1} yields $(\forall\varepsilon\in\left]0,2\right[)(\exists m_\varepsilon\in\mathbb{N})(\forall\gamma\in\left]0,2-\varepsilon\right])$
    \begin{subequations}
        \begin{align}
            \sum_{n=0}^{\infty}\norm{x^{(n+1)}-x^{(n)}}^\gamma & \geq\sum_{n=m_\varepsilon}^{\infty}\norm{x^{(2n)}-x^{(2n-1)}}^\gamma\notag                                      \\
                                                               & =\sum_{n=m_\varepsilon}^{\infty}\exp\left(-\gamma\beta u_n^{-r}\right)\tag{by \cref{p:R2sequence1} }            \\
                                                               & \geq\sum_{n=m_\varepsilon}^{\infty}\exp\left(-\gamma\frac{\ln(n)}{2-\varepsilon}\right)\tag{by \cref{260106a} } \\
                                                               & =\sum_{n=m_\varepsilon}^{\infty}n^{-\frac{\gamma}{2-\varepsilon}}=+\infty;\notag
        \end{align}
    \end{subequations}
    consequently,
    \begin{equation}
        \sum_{n=0}^{\infty}\norm{x^{(n+1)}-x^{(n)}}^\gamma=+\infty \quad
        \text{for all $\gamma\in\left]0,2\right[$.}
    \end{equation}
\end{example}

\begin{remark}
    In \cite[Example 4.10]{HM2016}, the authors also study the asymptotic behavior of $u_n$ in \cref{e:notcone} with $\beta=1$ and $r=2$. They showed that $\frac{1}{u_n}= O\left(n^{1/t}\right)$ for all $t\in\mathbb{R}_{++}$. Using \cref{p:R2sequence2}, we obtain $\frac{1}{u_n}\sim \sqrt{\frac{\ln(n)}{2}}$, which is a sharper result.
\end{remark}

\section{Conification and limiting examples for nonpolyhedral cones in $\mathbb{R}^3$}

\label{s:last}

In this section, we will construct a family of collections, each containing one polyhedral cone and one nonpolyhedral cone, such that \cref{t:GTpolyconeinfi} fails. The construction is based on the conification of \cref{e:notcone}.
Because the counterexample is in $\RR^3$, we note that it is optimal in the sense that such an example
cannot be found in $\RR$ or $\RR^2$, where all closed convex cones
are automatically polyhedral.

{ We will require the following result on the projection of the conification of a closed convex set
in terms of its recession cone\footnote{The \emph{recession cone} of $C$ is defined by $\rec(C)=\{x\in X\mid x+C\subseteq C\}$.}
(see \cite[Theorem~4.1 and 5.1]{BTH2023}):

\begin{fact}\label{f:hcone}
    Let $C$ be a closed convex subset of $X$ such that $0\in C$.
    For all $(y,s)\in X\times\RR$, the function
    \begin{align}
        \Psi_{(y,s)}:\RR\to\left]-\infty,+\infty\right]:\alpha\mapsto\begin{cases}
                                                                         +\infty,                                 & \text{ if }\alpha<0; \\
                                                                         d^2_{\rec (C)}\left(y\right)+s^2,        & \text{ if }\alpha=0; \\
                                                                         \alpha^2 d_{C}^2(y/\alpha)+(\alpha-s)^2, & \text{ if }\alpha>0,
                                                                     \end{cases}
    \end{align}
    is strictly convex and supercoercive\footnote{
        Recall that a function $f\colon X\to[-\infty,+\infty]$ is \emph{supercoercive} if
        $\lim_{\norm{x}\to+\infty}\frac{f(x)}{\norm{x}}=+\infty$.
    }, with unique minimizer $\alpha^*\geq0$. Moreover, we have
    \begin{equation}
        (\forall\alpha>0)\quad\Psi_{(y,s)}'(\alpha)=-2\alpha\langle P_C(y/\alpha),(\Id-P_C)(y/\alpha)\rangle+2(\alpha-s).
    \end{equation}
    The projection onto $\ccone(C\times\{1\})$ is then given by
    \begin{align}
        P_{\ccone(C\times\{1\})}(y,s)=\begin{cases}
                                          \left(P_{\rec (C)}(y),0\right),                & \text{ if }\alpha^*=0; \\
                                          \left(\alpha^*P_C(y/\alpha^*),\alpha^*\right), & \text{ if }\alpha^*>0.
                                      \end{cases}
    \end{align}
\end{fact}

We will also need the following fact (see \cite[Proposition~9.29 and Example~9.32]{BC2017}):}
\begin{fact}\label{f:supercoercive}
    Let $f\colon X\to\left]-\infty,+\infty\right]$ be convex, lower semicontinuous, and proper. If $f$ is supercoercive, then
    $\rec\epi f=\epi\iota_{\{0\}} = \{0\}\times\RP\subseteq X\times\RR$,
    where $\iota$ is the indicator function.
\end{fact}

From now on, we consider the case
where our collection consists of
\begin{equation}
    H:=\mathbb{R}\times\{0\}\times\mathbb{R}\quad\text{and}\quad
    K:=\ccone(\epi f\times\{1\}) \subseteq \RR\times\RP\times\RP,
\end{equation}
where $f:\mathbb{R}\to\left]-\infty,+\infty\right]$ is
\begin{equation}
    \text{convex, differentiable on $[0,\delta]$, $\delta\in\mathbb{R}_{++}$, lower semicontinuous, and proper,}
\end{equation}
with
\begin{equation}
    f\text{ even, }f(0)=0,f>0\text{ otherwise, and }f'(0)=0.
\end{equation}

The projection onto $H$ is simple:
\begin{equation}
    P_H\colon \RR^3\to\RR^3 \colon (x_1,x_2,x_3)\mapsto(x_1,0,x_3).\label{251215a}
\end{equation}

{Let $b_0\in\left]0,\min\{1,\delta\}\right[$, $a_0>0$, and $x^{(0)}:=a_0(b_0,0,1)\in H$.} Generate the sequence of alternating projections via
\begin{equation}
    x^{(2n+1)} := P_{K}x^{(2n)}\text{ and }x^{(2n+2)} := P_{H}x^{(2n+1)}.\label{251215b}
\end{equation}
{ An instance of the trajectory is illustrated in \cref{fig:image2}.}

\begin{figure}[H]
    \label{fig:3D}
    \begin{subfigure}{0.47\linewidth}
        \includegraphics[width=\linewidth]{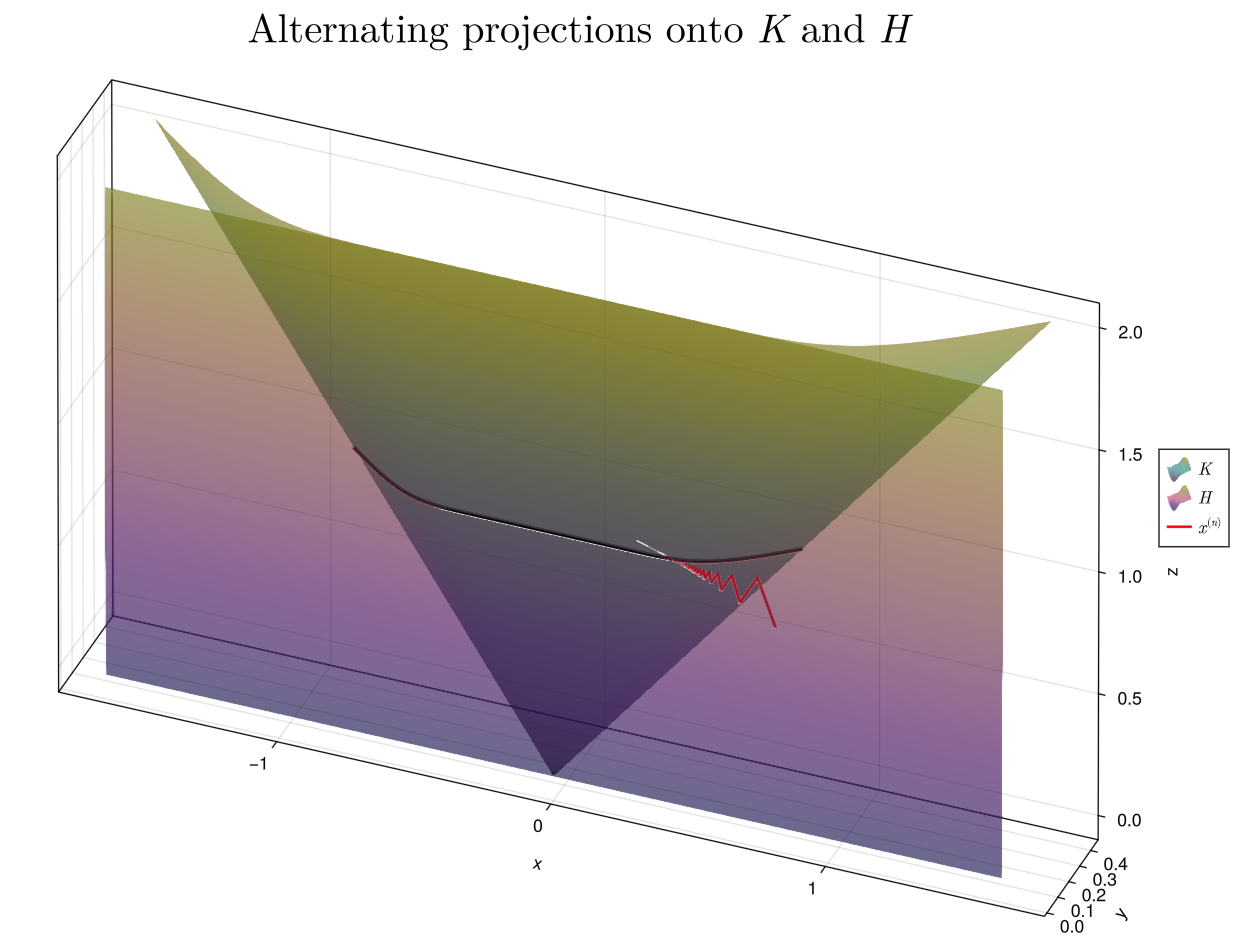}
        \label{fig:subim1}
    \end{subfigure}
    \begin{subfigure}{0.47\linewidth}
        \includegraphics[width=\linewidth]{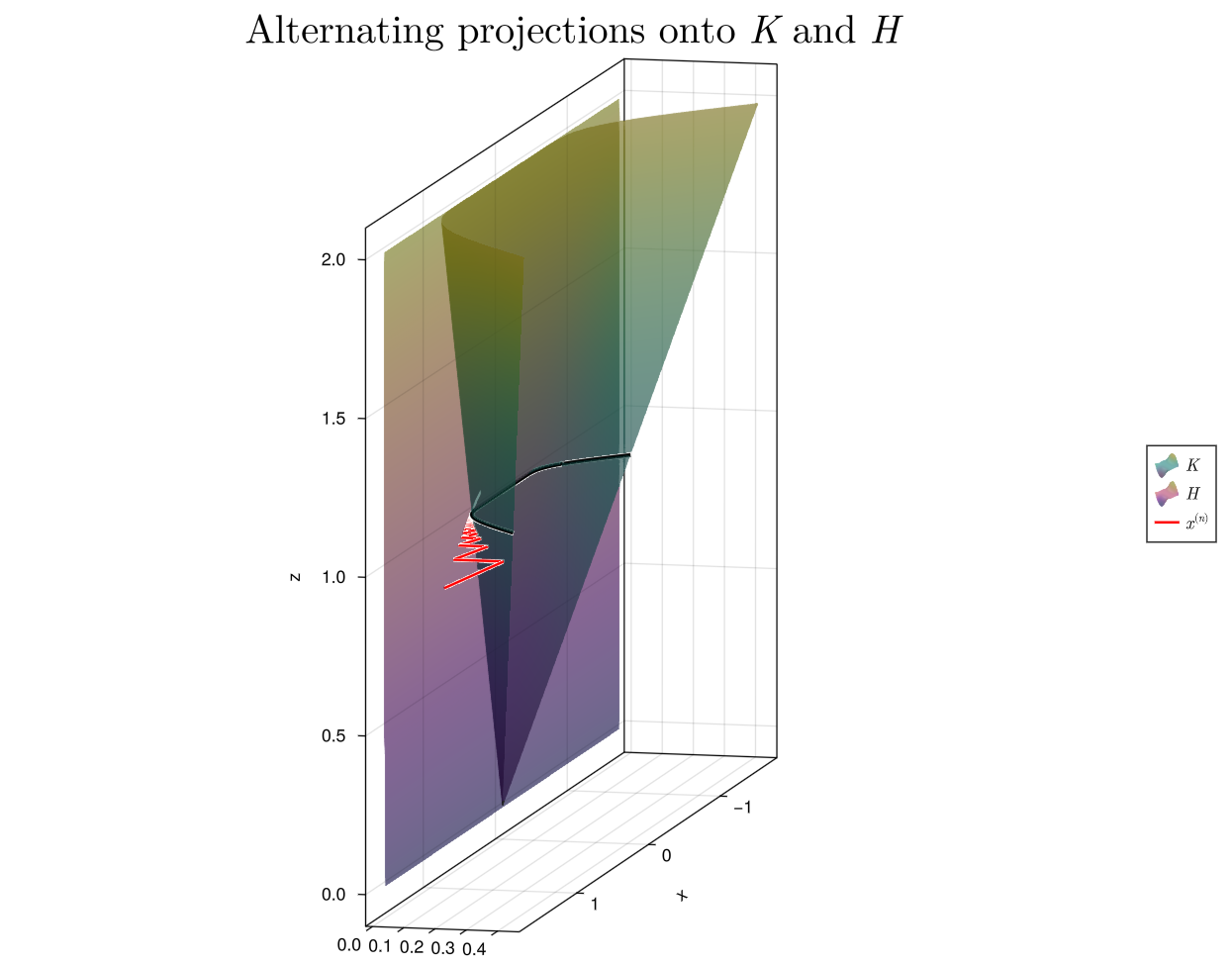}
        \label{fig:subim2}
    \end{subfigure}

    \caption{The trajectory of \cref{251215b} for $f(x)=\exp(- x^{-2})$ with domain $\bigl[-\sqrt{2/3},\sqrt{2/3}\bigr]$.}
    \label{fig:image2}
\end{figure}

\begin{proposition}\label{p:R3sequence1}
    Suppose that $f$ is supercoercive. Then
    \begin{enumerate}[ref=\currentenv~\theproposition(\roman*)]
        {
            \item\label{p:R3sequence1i-}
            $H \cap K = \{0\}\times\{0\}\times\RP$.        }
        \item\label{p:R3sequence1i} $x^{(n)} \to \bar{x} := \bar{a}(0,0,1) \in
                  (H \cap K)\smallsetminus\{0\}$.
        \item\label{p:R3sequence1ii} For all $n\in\mathbb{N}$: $0\notin\argmin\Psi_{x^{(2n)}}$.
        \item\label{p:R3sequence1iii} For all $n\in\mathbb{N}\smallsetminus\{0\}$, the iterates have the following form
              \begin{equation}
                  x^{(2n-1)}=a_n\left(b_n,f(b_n),1\right)\quad\text{and}\quad x^{(2n)}=a_n(b_n,0,1),
              \end{equation}
              where
              \begin{equation}
                  \{a_{n}\}=\argmin\Psi_{x^{(2n-2)}}\quad\text{and}\quad \left(b_{n},f(b_n)\right)=P_{\epi f}\Big(\frac{a_{n-1}}{a_{n}}(b_{n-1},0)\Big),
              \end{equation}
              and $a_n>0$, $b_n\in\left]0,\delta\right[$. Moreover,
              \begin{equation}
                  (\forall n\in\mathbb{N}\smallsetminus\{0\})\quad \frac{a_{n-1}}{a_{n}}b_{n-1}=b_{n}+f(b_n)f'(b_n).
              \end{equation}
        \item\label{p:R3sequence1iv} For all $n\in\mathbb{N}\smallsetminus\{0\}$:
              \begin{equation}
                  \frac{a_{n-1}}{a_n}=1+f(b_n)\big(f(b_n)-f'(b_n)b_n\big).
              \end{equation}
              Moreover,
              \begin{equation}
                  (\forall n\in\mathbb{N})\quad 0<a_{n}\leq a_{n+1}\quad\text{and}\quad 0<b_{n+1}<\frac{a_n}{a_{n+1}}b_n\leq b_n,
              \end{equation}
              and
              \begin{equation}
                  a_n\to\bar{a}>0,\;\; b_n\to 0.
              \end{equation}
    \end{enumerate}
\end{proposition}
\begin{proof}
    {
        {\it (i)}:
        \cref{f:supercoercive} in tandem with
        \cite[Corollary~6.53]{BC2017} yield
        \begin{equation}
            \label{e:251231b}
            K = {\ccone(\epi f \times \{1\})} = \menge{\alpha(x,y,1)}{\alpha>0, f(x)\leq y} \cup
            \big(\{0\}\times\RP\times \{0\}\big).
        \end{equation}
        Let $(u,v,w)\in H\cap K$. Because $(u,v,w)\in H$, we have $v=0$.
        Suppose first that there exists $\alpha>0$ such that
        $u=\alpha x, v=\alpha y, w=\alpha$ and $f(x)\leq y$.
        Now $v=0$, so $y=0$ which yields $f(x)=0$, i.e., $x=0$.
        Thus $(u,v,w) \in \{0\}\times\{0\}\times\RPP$.
        Now suppose that $(u,v,w)\in\{0\}\times\RP\times\{0\}$.
        Then $(u,v,w)=(0,0,0)$.
        In either case, $(u,v,w)\in\{0\}\times\{0\}\times\RP$ and so
        \begin{equation}
            H\cap K \subseteq \{0\}\times\{0\}\times\RP.
        \end{equation}
        Conversely, let $(0,0,\beta)\in \{0\}\times\{0\}\times\RP$.
        Clearly, $(0,0,\beta)\in H$.
        If $\beta=0$, then $(0,0,\beta)=(0,0,0) \in \{0\}\times \RP \times\{0\}
            \subseteq K$ by \cref{e:251231b}.
        And if $\beta>0$, then $(0,0,\beta) = \beta(0,f(0),1) \in K$, again
        by \cref{e:251231b}.
        Altogether, $(0,0,\beta)\in H\cap K$ and we've verifed \cref{p:R3sequence1i-}.
    }

    {\it (ii)}: \cref{p:R3sequence1i-} and
    \cite[Theorem~1]{Bregman1965} imply
    $x^{(n)} \to \bar{x} = \bar{a}(0,0,1) \in H \cap K=
        \{0\}\times\{0\}\times\RP$.
    Since $a_0(0,0,1)\in H\cap K$ and $(x^{(n)})_{n\in\NN}$ is Fej\'er monotone with respect to $H\cap K$ we obtain
    \begin{equation}
        (\forall n\in\NN)\quad d_{\{a_0(0,0,1)\}}\big(x^{(n)}\big)\leq d_{\{a_0(0,0,1)\}}\big(x^{(0)}\big)=a_0b_0.
    \end{equation}
    By the continuity of the distance function, we get that
    \begin{equation}
        |\bar{a}-a_0|=d_{\{a_0(0,0,1)\}}\left(\bar{x}\right)\leq a_0b_0.
    \end{equation}
    Since $a_0>0$ and $b_0<1$ by assumption, this implies $\bar{a}>0$.

        {\it (iii)}: Since $f$ is supercoercive, \cref{f:supercoercive} yields
    $\rec\epi f=\{0\}\times\RP$.
    By combining this with \cref{251215b} and \cref{f:hcone}, we get that
    \begin{subequations}
        \begin{align}
            (\forall n\in\mathbb{N})\quad x^{(2n+1)}=P_Kx^{(2n)} & =\begin{cases}
                                                                        \left(P_{\rec (\epi f)}\left(x^{(2n)}_1,x^{(2n)}_2\right),0\right),                                        & \text{ if }\alpha^*=0; \\
                                                                        \left(\alpha^*P_{\epi f}\left(\frac{1}{\alpha^*}\left(x^{(2n)}_1,x^{(2n)}_2\right)\right),\alpha^*\right), & \text{ if }\alpha^*>0,
                                                                    \end{cases} \\
                                                                 & =\begin{cases}
                                                                        \left(0,\max\left\{x^{(2n)}_2,0\right\},0\right),                                                          & \text{ if }\alpha^*=0; \\
                                                                        \left(\alpha^*P_{\epi f}\left(\frac{1}{\alpha^*}\left(x^{(2n)}_1,x^{(2n)}_2\right)\right),\alpha^*\right), & \text{ if }\alpha^*>0,
                                                                    \end{cases}
        \end{align}
    \end{subequations}
    where $\{\alpha^*\}=\amin \Psi_{x^{(2n)}}$. By \cref{251215a}, we have $P_H\big(0,
        \max\{x^{(2n)}_2,0\},0\big)=(0,0,0)$. It then follows from \cref{p:R3sequence1i} that $0\notin \amin \Psi_{x^{(2n)}}$ for all $n\in\mathbb{N}$.

        {\it(iv)}: We will prove this by induction. For the base case $n=1$, by \cref{251215b}, \cref{p:R3sequence1ii}, and \cref{f:hcone}, we have
    \begin{equation}
        x^{(1)}=P_{K}x^{(0)}=a_1\bigg(P_{\epi f}\Big(\frac{a_0}{a_{1}}(b_0,0)\Big),1\bigg)=a_1\big(b_1,f(b_1),1\big),
    \end{equation}
    where
    \begin{equation}
        \RR_{++}\supset\{a_{1}\}=\argmin\Psi_{x^{(0)}}\quad\text{and}\quad \big(b_{1},f(b_1)\big)=P_{\epi f}\Big(\frac{a_0}{a_{1}}(b_0,0)\Big).
    \end{equation}

    {Note that by \cref{f:hcone}, $f(0)=0$, and \cite[Theorem~3.16]{BC2017}, we have that
        \begin{subequations}
            \begin{align}
                \Psi'_{x^{(0)}}(a_0)
                 & =-2a_0\left\langle P_{\epi f}\Big(\frac{a_0}{a_0}(b_0,0)\Big),(\Id-P_{\epi f})\Big(\frac{a_0}{a_0}(b_0,0)\Big)\right\rangle +2(a_0-a_0) \\
                 & =2a_0\left\langle(0,0)- P_{\epi f}\big((b_0,0)\big),(\Id-P_{\epi f})\big((b_0,0)\big)\right\rangle                                      \\
                 & \leq0.
            \end{align}
        \end{subequations}
        This combined with $\Psi_{x^{(0)}}$ is strictly convex and $\{a_{1}\}=\argmin\Psi_{x^{(0)}}$ yields $a_0\leq a_1$. Since $a_0\in\mathbb{R}_{++}$ and $b_0\in\left]0,\delta\right[$ by assumption, we obtain $\frac{a_0}{a_1}b_0\leq b_0<\delta$. Hence, \cref{f:epi} then yields $b_1\in\left]0,\delta\right[$ and
        \begin{equation}
            \frac{a_0}{a_1}b_0=b_1+f(b_1)f'(b_1).
        \end{equation}}
    Using \cref{251215a} and \cref{251215b}, we obtain $x^{(2)}=P_Hx^{(1)}=a_1(b_1,0,1)$.

    Now assume that for some $k\in\mathbb{N}\smallsetminus\{0\}$, we have
    \begin{equation}
        x^{(2k-1)}=a_k\big(b_k,f(b_k),1\big)\quad\text{and}\quad x^{(2k)}=a_k(b_k,0,1),
    \end{equation}
    where $a_k>0$ and $b_k\in\left]0,\delta\right[$. By \cref{251215b}, \cref{p:R3sequence1ii}
    and \cref{f:hcone}, we have
    \begin{equation}
        x^{(2k+1)}=P_{K}x^{(2k)}=a_{k+1}\left(P_{\epi f}\Big(\frac{a_k}{a_{k+1}}(b_k,0)\Big),1\right)
        =a_{k+1}\big(b_{k+1},f(b_{k+1}),1\big),
    \end{equation}
    where
    \begin{equation}
        \RR_{++}\supset\{a_{k+1}\}=\argmin\Psi_{x^{(2k)}}\quad\text{and}\quad \big(b_{k+1},f(b_{k+1})\big)=P_{\epi f}\Big(\frac{a_k}{a_{k+1}}(b_k,0)\Big).
    \end{equation}

    {Note that by \cref{f:hcone}, $f(0)=0$, and \cite[Theorem~3.16]{BC2017}, we have that
    \begin{subequations}
        \begin{align}
            \Psi'_{x^{(2k)}}(a_{k})
             & =-2a_{k}\left\langle P_{\epi f}\Big(\frac{a_{k}}{a_{k}}(b_{k},0)\Big),(\Id-P_{\epi f})\Big(\frac{a_{k}}{a_{k}}(b_{k},0)\Big)\right\rangle +2(a_{k}-a_{k}) \\
             & =2a_{k}\left\langle(0,0)- P_{\epi f}\big((b_{k},0)\big),(\Id-P_{\epi f})\big((b_{k},0)\big)\right\rangle                                                  \\
             & \leq0.
        \end{align}
    \end{subequations}
    This combined with $\Psi_{x^{(2k)}}$ is strictly convex and $\{a_{k+1}\}=\argmin\Psi_{x^{(2k)}}$ yields $a_k\leq a_{k+1}$. Since $a_k\in\mathbb{R}_{++}$ and $b_k\in\left]0,\delta\right[$ by the induction hypothesis, we obtain $\frac{a_k}{a_{k+1}}b_k\leq b_k<\delta$. Hence, \cref{f:epi} then yields $b_{k+1}\in\left]0,\delta\right[$ and
    \begin{equation}
        \frac{a_k}{a_{k+1}}b_k=b_{k+1}+f(b_{k+1})f'(b_{k+1}).
    \end{equation}}
    Using \cref{251215a} and \cref{251215b}, we obtain $x^{(2k+2)}=P_Hx^{(2k+1)}=a_{k+1}(b_{k+1},0,1)$.

        {\it(v)}: By \cref{p:R3sequence1iii} and \cref{f:hcone}, we get that $0<a_n\in\argmin\Psi_{(a_{n-1}b_{n-1},0,a_{n-1})}$ for all $n\in\mathbb{N}\setminus\{0\}$ and $\Psi_x$ is convex for all $x\in \mathbb{R}^3$. This implies that for all $n\in\NN\smallsetminus\{0\}$
    \begin{subequations}
        \begin{align}
            0
            =
            \qquad\ \:  & \Psi'_{(a_{n-1}b_{n-1},0,a_{n-1})}(a_n)
            \\
            \eqoverset{\text{\cref{f:hcone}}}{=}
            \quad\ \ \  & -2a_n\left\langle P_{\epi f}
            \Big(\frac{a_{n-1}b_{n-1}}{a_n},0\Big),(\Id-P_{\epi f})\Big(\frac{a_{n-1}b_{n-1}}{a_n},0\Big)\right\rangle +2(a_n-a_{n-1})    \\
            \eqoverset{\text{\cref{p:R3sequence1iii}}}{=}
            \,          & -2a_n\left\langle \big(b_n,f(b_n)\big),\Big(\frac{a_{n-1}b_{n-1}}{a_n},0\Big)-\big(b_n,f(b_n)\big)\right\rangle
            +2(a_n-a_{n-1})                                                                                                               \\
            \eqoverset{\text{\cref{p:R3sequence1iii}}}{=}
            \,          & -2a_n\left\langle \big(b_n,f(b_n)\big),\left(f(b_n)f'(b_n),-f(b_n)\right)\right\rangle+2(a_n-a_{n-1})           \\
            =
            \qquad\ \:  & -2a_n\big(b_nf(b_n)f'(b_n)-f^2(b_n)\big)+2(a_n-a_{n-1})
            \\
            =
            \qquad\ \:  & -2a_nf(b_n)\big(b_nf'(b_n)-f(b_n)\big)+2(a_n-a_{n-1}).
        \end{align}
    \end{subequations}
    Switch sides, simplify, and apply the convexity of $f$, we obtain
    \begin{equation}
        0<\frac{a_{n-1}}{a_n}=1+f(b_n)\left(f(b_n)+f'(b_n)(0-b_n)\right)\leq1+f(b_n)f(0)=1.\label{251215c1}
    \end{equation}
    From \cref{p:R3sequence1iii}, we get that $\left(b_{n},f(b_n)\right)=P_{\epi f}\big(\frac{a_{n-1}}{a_{n}}(b_{n-1},0)\big)$, and $b_n>0$ for all $n\in\mathbb{N}\smallsetminus\{0\}$. \cref{f:epi} and \cref{251215c1} then yield $b_{n}\in\bigl]0,\frac{a_{n-1}}{a_{n}}b_{n-1}\bigr[\subseteq\left]0,b_{n-1}\right[$ for all $n\in\mathbb{N}\smallsetminus\{0\}$.
    The convergence $a_n\to\bar{a}$ and $b_n\to0$ follows from \cref{p:R3sequence1i} and \cref{p:R3sequence1iii}.
\end{proof}

Next, we analyze the asymptotic behavior of $(b_n)_{n\in\mathbb{N}}$.
\begin{proposition}\label{p:R3sequence2}
    Suppose that $f$ is supercoercive and there exists $q\in\mathbb{R}$ and $p,\alpha\in\mathbb{R}_{++}$ such that
    \begin{equation}
        \lim_{x\searrow0}\frac{f(x)f'(x)}{x^q\exp(-\alpha x^{-p})}=c_{q,\alpha,p}>0.\label{251215c}
    \end{equation}
    Then $\displaystyle \frac{b_n}{\left(\frac{\alpha}{\ln(n)}\right)^{1/p}}\to1$ as $n\to\infty$.
\end{proposition}
\begin{proof}
    By \cref{p:R3sequence1iii} and \cref{p:R3sequence1iv}, we obtain $0<a_n\leq a_{n+1}\to\bar{a}$, $b_n\geq\frac{a_n}{a_{n+1}}b_n>b_{n+1}\to0^+$, and
    \begin{equation}
        (\forall n\in\mathbb{N}\smallsetminus\{0\})\quad \frac{a_{n-1}}{a_n}b_{n-1}=b_n+f(b_n)f'(b_n).\label{251215e}
    \end{equation}
    This yields
    \begin{equation}
        (\forall n\in\mathbb{N}\smallsetminus\{0\})\quad \frac{f(b_n)f'(b_n)}{b_n^q\exp(-\alpha b_n^{-p})}=\frac{\frac{a_{n-1}}{a_n}b_{n-1}-b_n}{b_n^q\exp(-\alpha b_n^{-p})}.\label{251215f}
    \end{equation}
    By checking the derivative,
    $\displaystyle \frac{1}{x^q\exp(-\alpha x^{-p})}$ is decreasing on $\mathbb{R}_{++}$ if $q\geq0$ and on $\Bigl]0,\big(\frac{-q}{\alpha p}\big)^{-1/p}\Bigr[$ if $q<0$. WLOG, assume that $b_0$ lies inside the decreasing interval. Since $\frac{a_{n-1}}{a_{n}}b_{n-1}>b_{n}$, we get
    \begin{subequations}\label{251215i}
        \begin{align}
            (\forall n\in\mathbb{N}\smallsetminus\{0\})\quad\frac{\frac{a_{n-1}}{a_{n}}b_{n-1}-b_n}{b_n^q\exp(-\alpha b_n^{-p})} & \geq\int_{b_n}^{\frac{a_{n-1}}{a_{n}}b_{n-1}}\frac{dx}{x^q\exp(-\alpha x^{-p})}                                                                            \\
                                                                                                                                 & \geq\frac{\frac{a_{n-1}}{a_{n}}b_{n-1}-b_n}{\left(\frac{a_{n-1}}{a_{n}}b_{n-1}\right)^q\exp\Big(-\alpha \big(\frac{a_{n-1}}{a_{n}}b_{n-1}\big)^{-p}\Big)}.
        \end{align}
    \end{subequations}
    It follows from $b_n\searrow0^+$, \cref{251215c}, and \cref{251215f} that
    \begin{equation}
        \lim_{n\to\infty}\frac{\frac{a_{n-1}}{a_{n}}b_{n-1}-b_n}{b_n^q\exp(-\alpha b_n^{-p})}=c_{q,\alpha,p}>0.\label{251215g}
    \end{equation}
    Observe that \cref{251215c} and \cref{251215e} imply
    \begin{subequations}
        \label{251215h}
        \begin{align}
            \lim_{n\to\infty}\frac{\frac{a_{n-1}}{a_{n}}b_{n-1}}{b_n} & =1+\lim_{n\to\infty}\frac{f(b_n)f'(b_n)}{b_n}                                                                \\
                                                                      & =1+\lim_{n\to\infty}\frac{f(b_n)f'(b_n)}{b_n^q\exp(-\alpha b_n^{-p})}\frac{b_n^q\exp(-\alpha b_n^{-p})}{b_n} \\
                                                                      & =1+c_{q,\alpha,p}\lim_{n\to\infty}\frac{b_n^{q-1}}{\exp(\alpha b_n^{-p})}                                    \\
                                                                      & =1.
        \end{align}
    \end{subequations}
    By combining \cref{251215g} and \cref{251215h}, we obtain
    \begin{subequations}
        \begin{align}
                & \lim_{n\to\infty}\frac{\frac{a_{n-1}}{a_{n}}b_{n-1}-b_n}{\left(\frac{a_{n-1}}{a_{n}}b_{n-1}\right)^q\exp\left(-\alpha \left(\frac{a_{n-1}}{a_{n}}b_{n-1}\right)^{-p}\right)}                                                                        \\
            =\  & \lim_{n\to\infty}\frac{\frac{a_{n-1}}{a_{n}}b_{n-1}-b_n}{b_{n}^q\exp(-\alpha b_{n}^{-p})}\frac{b_{n}^q\exp(-\alpha b_{n}^{-p})}{\left(\frac{a_{n-1}}{a_{n}}b_{n-1}\right)^q\exp\left(-\alpha \left(\frac{a_{n-1}}{a_{n}}b_{n-1}\right)^{-p}\right)} \\
            =\  & c_{q,\alpha,p}\lim_{n\to\infty}\frac{b_n^q}{\left(\frac{a_{n-1}}{a_{n}}b_{n-1}\right)^q}\exp\left(-\alpha b_n^{-p}\left(1-\frac{\left(\frac{a_{n-1}}{a_{n}}b_{n-1}\right)^{-p}}{b_n^{-p}}\right)\right)                                             \\
            =\  & c_{q,\alpha,p}\exp\left(-\alpha\lim_{n\to\infty} b_n^{-p}\left(1-\frac{b_{n}^{p}}{\left(\frac{a_{n-1}}{a_{n}}b_{n-1}\right)^{p}}\right)\right).
        \end{align}
    \end{subequations}
    By the Taylor expansion of $1-x^p$ at $x=1$, we get that $\frac{1-x^p}{p(1-x)}\to 1$ as $x\to1$. This together with \cref{251215h} yields
    \begin{equation}
        \frac{1-\frac{b_{n}^{p}}{\left(\frac{a_{n-1}}{a_{n}}b_{n-1}\right)^{p}}}{p\left(1-\frac{b_n}{\frac{a_{n-1}}{a_{n}}b_{n-1}}\right)}\to1.
    \end{equation}
    Hence,
    \begin{subequations}
        \label{251215j}
        \begin{align}
                  & \lim_{n\to\infty}\frac{\frac{a_{n-1}}{a_{n}}b_{n-1}-b_n}{\left(\frac{a_{n-1}}{a_{n}}b_{n-1}\right)^q\exp\left(-\alpha \left(\frac{a_{n-1}}{a_{n}}b_{n-1}\right)^{-p}\right)} \\
            =\:\, & c_{q,\alpha,p}\exp\left(-\alpha\lim_{n\to\infty} b_n^{-p}\left(1-\frac{b_{n}^{p}}{\left(\frac{a_{n-1}}{a_{n}}b_{n-1}\right)^{p}}\right)\right)                               \\
            =\:\, & c_{q,\alpha,p}\exp\left(-\alpha p\lim_{n\to\infty} b_n^{-p}\left(1-\frac{b_{n}}{\frac{a_{n-1}}{a_{n}}b_{n-1}}\right)\right)                                                  \\
            \eqoverset{\text{\cref{251215h}}}{=}\
                  & c_{q,\alpha,p}\exp\left(-\alpha p\lim_{n\to\infty} \frac{\frac{a_{n-1}}{a_{n}}b_{n-1}-b_{n}}{b_{n}^{p+1}}\right)                                                             \\
            \eqoverset{\text{\cref{251215g}}}{=}
            \     & c_{q,\alpha,p}\exp\left(-\alpha pc_{q,\alpha,p}\lim_{n\to\infty} \frac{b_n^q\exp(-\alpha b_n^{-p})}{b_{n}^{p+1}}\right)                                                      \\
            =\:\, & c_{q,\alpha,p}\exp\left(0\right)=c_{q,\alpha,p}.
        \end{align}
    \end{subequations}
    From \cref{251215i}, \cref{251215g}, and \cref{251215j}, we obtain
    \begin{equation}
        \lim_{n\to\infty}\int_{b_n}^{\frac{a_{n-1}}{a_{n}}b_{n-1}}\frac{dx}{x^q\exp(-\alpha x^{-p})}=c_{q,\alpha,p}.\label{251215o}
    \end{equation}
    Since $\displaystyle \frac{1}{x^q\exp(-\alpha x^{-p})}$ is decreasing on
    $\big[\frac{a_{n-1}}{a_{n}}b_{n-1},b_{n-1}\big]$, we get that
    \begin{subequations}\label{251215k}
        \begin{align}
            0 & \leq\int_{\frac{a_{n-1}}{a_{n}}b_{n-1}}^{b_{n-1}}\frac{dx}{x^q\exp(-\alpha x^{-p})}                                                                                      \\
              & \leq\left(b_{n-1}-\frac{a_{n-1}}{a_{n}}b_{n-1}\right)\left(\frac{a_{n-1}}{a_{n}}b_{n-1}\right)^{-q}\exp\left(\alpha\left(\frac{a_{n-1}}{a_{n}}b_{n-1}\right)^{-p}\right) \\
              & = b_{n-1}f(b_n)\left(f'(b_n)b_n-f(b_n)\right)\left(\frac{a_{n-1}}{a_{n}}b_{n-1}\right)^{-q}
            \exp\left(\alpha\left(\frac{a_{n-1}}{a_{n}}b_{n-1}\right)^{-p}\right),
        \end{align}
    \end{subequations}
    where we used \cref{p:R3sequence1iv}
    in the last equality.
    By \cref{p:R3sequence1iv}, we have $\frac{a_{n-1}}{a_n}\to\frac{\bar{a}}{\bar{a}}=1$. This combined with \cref{251215h} yields
    \begin{equation}
        \frac{b_{n-1}}{b_n}\to1.\label{251215l}
    \end{equation}
    We also get from \cref{251215j} that
    \begin{subequations}
        \label{251215m}
        \begin{align}
            \lim_{n\to\infty}\frac{\exp\left(\alpha\left(\frac{a_{n-1}}{a_{n}}b_{n-1}\right)^{-p}\right)}{\exp\left(\alpha b_n^{-p}\right)}
             & =\exp\left(-\alpha\lim_{n\to\infty}b_n^{-p}\left(1-\frac{b_n^p}{\left(\frac{a_{n-1}}{a_{n}}b_{n-1}\right)^{p}}\right)\right) \\
             & =\exp(0)=1.
        \end{align}
    \end{subequations}
    By combining \cref{251215l}, \cref{251215h}, and \cref{251215m}, we obtain
    \begin{subequations}
        \begin{align}
                & \lim_{n\to\infty}b_{n-1}f(b_n)\left(f'(b_n)b_n-f(b_n)\right)\left(\frac{a_{n-1}}{a_{n}}b_{n-1}\right)^{-q}\exp\left(\alpha\left(\frac{a_{n-1}}{a_{n}}b_{n-1}\right)^{-p}\right) \\
            =\  & \lim_{n\to\infty}b_{n}f(b_n)\left(f'(b_n)b_n-f(b_n)\right)b_n^{-q}\exp\left(\alpha b_n^{-p}\right)                                                                              \\
            =\  & \lim_{n\to\infty}b_n^2\left(\frac{f(b_n)f'(b_n)}{b_n^{q}\exp\left(-\alpha b_n^{-p}\right)}-\frac{f^2(b_n)}{b_n^{q+1}\exp\left(-\alpha b_n^{-p}\right)}\right).
        \end{align}
    \end{subequations}
    Since $b_n\searrow 0^+$, by assumption, we obtain $\frac{f(b_n)f'(b_n)}{b_n^{q}\exp\left(-\alpha b_n^{-p}\right)}\to c_{q,\alpha,p}>0$. Moreover, by L'Hôpital's rule, we also get
    \begin{subequations}
        \begin{align}
            \lim_{n\to\infty} \frac{f^2(b_n)}{b_n^{q+1}\exp\left(-\alpha b_n^{-p}\right)} & =\lim_{n\to\infty} \frac{2f(b_n)f'(b_n)}{(q+1)b_n^{q}\exp\left(-\alpha b_n^{-p}\right)+\alpha pb_n^{q-p}\exp\left(-\alpha b_n^{-p}\right)} \\
                                                                                          & =\lim_{n\to\infty} \frac{2f(b_n)f'(b_n)}{b_n^{q}\exp\left(-\alpha b_n^{-p}\right)\left(q+1+\alpha pb_n^{-p}\right)}                        \\
                                                                                          & =0.
        \end{align}
    \end{subequations}
    Altogether, we have
    \begin{equation}
        \lim_{n\to\infty}b_{n-1}f(b_n)\left(f'(b_n)b_n-f(b_n)\right)\left(\frac{a_{n-1}}{a_{n}}b_{n-1}\right)^{-q}\exp\left(\alpha\left(\frac{a_{n-1}}{a_{n}}b_{n-1}\right)^{-p}\right)=0.\label{251215n}
    \end{equation}
    It then follows from \cref{251215k} and \cref{251215n} that
    \begin{equation}
        \lim_{n\to\infty}\int_{\frac{a_{n-1}}{a_{n}}b_{n-1}}^{b_{n-1}}\frac{dx}{x^q\exp(-\alpha x^{-p})}=0.
    \end{equation}
    This combined with \cref{251215o} yields
    \begin{equation}
        \lim_{n\to\infty}\int_{b_n}^{b_{n-1}}\frac{dx}{x^q\exp(-\alpha x^{-p})}=c_{q,\alpha,p}.
    \end{equation}
    The remainder of the proof is identical to the corresponding part of the proof of \cref{p:R2sequence2}, starting at \cref{260101a}.
\end{proof}

\begin{example}[G\"unt\"urk-Thao fails for a plane
        and a nonpolyhedral cone in $\RR^3$]
    \label{e:cone}
    Let $\beta\in \mathbb{R}_{++}$ and $r\in\{2,4,6,\ldots\}$.
    Consider the function $f(x)=\exp(-\beta x^{-r})$ with domain $\bigl[-\left(\beta r/(r+1)\right)^{1/r},\left(\beta r/(r+1)\right)^{1/r}\bigr]$. Note that $f$ is supercoercive and hence, the assumption of \cref{p:R3sequence2} is satisfied with $q=-r-1$, $\alpha=2\beta$, $p=r$, and $c_{q,\alpha,p}=r\beta$. Hence, we obtain
    \begin{equation}
        1=\lim_{n\to\infty}\frac{b_n}{\left(\frac{\alpha}{\ln(n)}\right)^{1/p}}=\lim_{n\to\infty}\frac{b_n}{\left(\frac{2\beta}{\ln(n)}\right)^{1/r}},
    \end{equation}
    which implies
    \begin{equation}
        (\forall\varepsilon\in\left]0,2\right[)(\exists m_\varepsilon\in\mathbb{N})(\forall n\geq m_\varepsilon)\quad b_n\geq\left(\left(1-\frac{\varepsilon}{2}\right)\left(\frac{2\beta}{\ln(n)}\right)\right)^{1/r}.\label{260106}
    \end{equation}
    This combined with \cref{p:R3sequence1} yields
    $ (\forall\varepsilon\left]0,2\right[)(\exists m_\varepsilon\in\mathbb{N})(\forall\gamma\in\left]0,2-\varepsilon\right])$
    \begin{subequations}
        \begin{align}
            \quad\sum_{n=0}^{\infty}\norm{x^{(n+1)}-x^{(n)}}^\gamma & \geq\sum_{n=m_\varepsilon}^{\infty}\norm{x^{(2n)}-x^{(2n-1)}}^\gamma
            \notag                                                                                                                                               \\
                                                                    & =\sum_{n=m_\varepsilon}^{\infty}
            a_n^\gamma\exp\left(-\gamma\beta b_n^{-r}\right)
            \tag{by \cref{p:R3sequence1iii} }                                                                                                                    \\
                                                                    & \geq\sum_{n=m_\varepsilon}^{\infty}
            a_0^\gamma\exp\left(-\gamma\frac{\ln(n)}{2-\varepsilon}\right)
            \tag{by \cref{p:R3sequence1iv} and \cref{260106} }                                                                                                   \\
                                                                    & =a_0^\gamma\sum_{n=m_\varepsilon}^{\infty}n^{-\frac{\gamma}{2-\varepsilon}}=+\infty;\notag
        \end{align}
    \end{subequations}
    consequently,
    \begin{equation}\sum_{n=0}^{\infty}\norm{x^{(n+1)}-x^{(n)}}^\gamma=+\infty \quad\text{for all $\gamma\in\left]0,2\right[$.}
    \end{equation}
\end{example}

\begin{remark}
    In this specific instance of the alternating projections sequence, it follows from \cref{p:R2sequence2} and \cref{p:R3sequence2} that the sequence $(u_n)_{n\in\mathbb{N}}$ and its conified counterpart $(b_n)_{n\in\mathbb{N}}$ have the same asymptotic behavior. One can also obtain a conified version of \cite[Proposition 4.5]{HM2016} by the same procedure used to conify \cref{p:R2sequence2} into \cref{p:R3sequence2}. The argument is simpler and largely repetitive, so we leave the details to the interested reader.
\end{remark}

\section*{Acknowledgments}
{The authors are grateful to the editor and two anonymous reviewers for helpful comments and suggestions which improved the manuscript.}
The research of HHB was supported by a Discovery Grant from the Natural Sciences and Engineering Research Council of Canada.

\end{document}